\documentclass[12pt]{amsart}
\usepackage{amsmath,amssymb}
\usepackage{amsthm}
\usepackage{color}
\newtheorem{theorem}{Theorem}[section]
\newtheorem{proposition}[theorem]{Proposition}

\newtheorem{lemma}[theorem]{Lemma}
\theoremstyle{definition}
\newtheorem{definition}[theorem]{Definition}

\makeatletter
\@addtoreset{equation}{section}
\makeatother

\newcommand{\CC}{{\mathbb C}} 
\newcommand{\NN}{{\mathbb N}} 
\newcommand{\ZZ}{{\mathbb Z}} 
\newcommand{\DD}{{\mathbb D}} 
\newcommand{\RR}{{\mathbb R}} 
\newcommand{\FF}{{\mathbb F}} 

\newcommand{\cB}{{\mathcal B}}

\newcommand{\cD}{{\mathcal D}}
\newcommand{\cE}{{\mathcal E}}
\newcommand{\cF}{{\mathcal F}}

\newcommand{\cH}{{\mathcal H}}

\newcommand{\cM}{{\mathcal M}}

\newcommand{\cX}{{\mathcal X}}

\newcommand{\iac}{\mathrm{i}} 
\newcommand{\de}{\mathop{}\!\mathrm{d}} 
\newcommand{\emath}{\mathrm{e}} 

\newcommand{\ra}{\rightarrow}
\newcommand{\ol}{\overline}
\renewcommand{\Re}{\mathrm{Re}}
\newcommand{\Span}{\operatorname{span}}
\providecommand{\AMS}{$\mathcal{A}$\kern-.1667em%
\lower.25em\hbox{$\mathcal{M}$}\kern-.125em$\mathcal{S}$}

\begin{document}
\title[Approximation of Functions in Reproducing Kernel Hilbert Spaces]{Probability Error Bounds for 
Approximation of Functions in Reproducing Kernel Hilbert Spaces} 

\author{Ata Den\.iz Ayd\i n}

\address{Department of Mathematics and Department of Computer Engineering, Bilkent University, 06800 Bilkent, Ankara,
  Turkey, \emph{current address:} ETH Z\"urich,
Department of Mathematics, R\"amistra\ss e 101,
8092 Z\"urich,
Switzerland}

\email{ata.aydin@alumni.bilkent.edu.tr}

\author{Aurelian Gheondea}

\address{Department of Mathematics, Bilkent University, 06800 Bilkent, Ankara,
Turkey \emph{and} Institutul de Matematic\u a al Academiei Rom\^ane, Calea Grivi\c tei 21,  010702 
Bucure\c sti, Rom\^ania} 
\email{aurelian@fen.bilkent.edu.tr \textrm{and} A.Gheondea@imar.ro}

\begin{abstract} We find probability error bounds for approximations of functions $f$ 
in a separable reproducing kernel Hilbert space $\cH$ with reproducing kernel $K$ on a base space $X$, 
firstly in terms of finite linear combinations of functions of type $K_{x_i}$ and then in terms of the projection
$\pi^n_x$ on $\Span\{K_{x_i} \}^n_{i=1}$, for random sequences of points $x=(x_i)_i$ in $X$.
Given a probability measure $P$, letting $P_K$ be the measure defined
by $\de P_K(x)=K(x,x)\de P(x)$, $x\in X$, our approach is based on the nonexpansive operator
\begin{equation*}
L^2(X;P_K)\ni \lambda\mapsto L_{P,K}\lambda:=\int_X \lambda(x)K_x\de P(x)\in \cH,
\end{equation*}  
where the integral exists in the Bochner sense. Using this operator, we then define a new reproducing kernel
Hilbert space, denoted by $\cH_P$, that is the operator range of $L_{P,K}$. 
Our main result establishes bounds, in terms of the operator $L_{P,K}$, 
on the probability that the Hilbert space distance between an arbitrary function $f$ in $\cH$ and 
linear combinations of functions of type $K_{x_i}$, for $(x_i)_i$ sampled independently from $P$, falls below a 
given threshold.
For sequences of points $(x_i)_{i=1}^\infty$ constituting a 
so-called uniqueness set, the orthogonal projections $\pi^n_x$ to 
$\Span\{K_{x_i} \}^n_{i=1}$ converge in the strong operator topology to the identity operator.
We prove that, under the assumption that $\cH_P$ is dense in $\cH$, any 
sequence of points sampled independently from $P$ yields a uniqueness set with probability $1$.
This result improves on previous error bounds in weaker norms, such as uniform or $L^p$ norms, which yield 
only convergence in probability and not almost certain convergence.
Two examples that show the applicability of this result to a uniform distribution on a compact interval and to
the Hardy space $H^2(\DD)$ are presented as well.
\end{abstract}
\maketitle

\section{Introduction}

Several machine learning algorithms that use positive semidefinite kernels, such as support vector 
machines (SVM), have 
been analysed and justified rigorously using the theory of reproducing kernel Hilbert spaces (RKHS), 
yielding statements of optimality, convergence and $L^p$ approximation bounds, e.g.\ see 
F.~Cucker and S.~Smale \cite{smale}. Reproducing kernel 
Hilbert spaces are Hilbert spaces of functions associated to a suitable kernel such that convergence with 
respect to the Hilbert space norm implies pointwise convergence, and in the context of approximation possess 
various favourable properties resulting from the Hilbert space structure. For example, under certain conditions 
on the kernel, every function in the Hilbert space is sufficiently differentiable and differentiation is in fact a 
nonexpansive linear map with respect to the Hilbert space norm, e.g.\ see 
\cite[Subsection 2.1.3]{saitoh}.

In order to substantiate the motivation for our investigation, we briefly review 
previously obtained bounds on the approximation of functions as linear combinations of kernels 
evaluated at finitely many points.
The theory of V.N.~Vapnik and A.Ya.~Chervonenkis of statistical learning theory 
\cite{vapnik}, \cite{vapnik1}, \cite{vapnikchervonenkis}, 
relies on concentration inequalities such as 
Hoeffding's inequality to bound the supremum distance between expected and empirical risk. The theory 
considers a data space $X \subseteq \RR^d$ on which an unknown probability distribution $P$ is 
defined, a hypothesis set $\cH$ and a loss function $V \colon \cH \times X \to \RR_+$, such 
that one wishes to find a hypothesis $h \in \cH$ that minimizes the expected risk 
\begin{equation*}R[h] := \int_X V(h, x) \de P(x).\end{equation*} 
Since $P$ is not known in general, instead of minimizing the expected risk one usually minimizes the 
empirical risk
\begin{equation*}\widehat R_S[h] = \frac{1}{n} \sum_{i=1}^n V(h, x_i)\end{equation*}
over a finite set $S = \{x_i\}_{i=1}^n \subseteq X$ of samples. Vapnik-Chervonenkis theory measures the 
probability with which the maximum distance between $R$ and $\widehat R$ falls below a given threshold.
Recall that the Vapnik-Chervonenkis (VC) dimension of $\cH$ with respect to $V$ is the maximum 
cardinality of finite subsets $Y \subseteq X$ that can be shattered by $\cH$, i.e.\ for each 
$Y^\prime \subseteq Y$, there exist $h \in \cH$ and $\alpha \in \RR$ such that
\begin{align*}
Y^\prime & = \left\{x \in Y \mid V(h, x) \geq \alpha \right\}; \\
Y \setminus Y^\prime & = \left\{x \in Y \mid V(h, x) < \alpha \right\}.
\end{align*}
Thus, they prove that, 
assuming that $A \leq V(h, x) \leq B$ for each $h \in \cH, x \in X$ and the VC dimension of $\cH$ is 
$d < \infty$, then, for any $\eta \in (0,1)$, 
\begin{equation*}
P\biggl(\sup_{h \in \cH} \left|R[h] - \widehat R_S[h]\right| \geq (B-A) \sqrt{\frac{d \log \frac{2en}{d} 
- \log\frac{\eta}{4}}{n}}\biggr) \leq \eta.
\end{equation*}

F.~Girosi, see \cite{girosi1} and \cite[Proposition 2]{sltapprox}, has used this general result to bound the 
uniform distance between 
integrals $\int J(x,y) \lambda(y) \de y$ and sums of the form $\frac{1}{n} \sum_{i=1}^n J(x, x_i)$, by 
reinterpreting $\cH$ as $\RR^d$, $V$ as $J$ and $\de P(y)$ as $\frac{|\lambda(y)|}{\|\lambda\|_{L^1}} \de y$.
M.A.~Kon and L.A.~Raphael \cite{sltapprox} then applied this methodology to obtain uniform 
approximation bounds of 
functions in reproducing kernel Hilbert spaces. They consider two cases where the Hilbert space is dense in 
$L^2(\RR^d)$ with a stronger norm \cite[Theorem 4]{sltapprox}, and where it is a closed subspace with the 
same norm \cite[Theorem 5]{sltapprox}. Also, M.A.~Kon, L.A.~Raphael, and 
D.A.~Williams \cite{kon} extended Girosi's approximation estimates for functions in Sobolev spaces.
While these bounds guarantee uniform convergence in probability, the approximating functions are neither
orthogonal projections of $f$ nor necessarily elements of a reproducing kernel Hilbert space, and hence may 
not capture $f$ exactly at $(x_i)_{i=1}^n$ nor converge monotonically. Furthermore, the fact that the norm is 
not a RKHS norm means that derivatives of $f$ may not be approximated in general, since differentiation is 
not bounded with respect to the uniform norm, unlike the RKHS norm associated to a continuously differentiable kernel.

The purpose of this article is thus to establish sufficient
conditions for convergence and approximation in the reproducing kernel Hilbert space norm.
In Section~\ref{s:mr},
we find probability error bounds for approximations of functions $f$ 
in a separable reproducing kernel Hilbert space $\cH$ with reproducing kernel $K$ on a base space $X$, 
firstly in terms of finite linear combinations of functions of type $K_x$ and then in terms of the projection
$\pi^n_x$ onto $\Span\{K_{x_i} \}^n_{i=1}$, for random 
sequences of points $x=(x_i)_i$ in the base space $X$. 
Given a probability measure $P$, letting $P_K$ be the measure defined
by $\de P_K(x)=K(x,x)\de P(x)$, $x\in X$, we approach these problems by firstly showing the existence of
the nonexpansive operator
\begin{equation}\label{e:lpk}
L^2(X;P_K)\ni \lambda\mapsto L_{P,K}\lambda:=\int_X \lambda(x)K_x\de P(x)\in \cH,
\end{equation}  
where the integral exists in the Bochner sense. Using this operator, we then define a new reproducing kernel
Hilbert space, denoted by $\cH_P$, that is the operator range of $L_{P,K}$.
Our main result establishes bounds, in terms of the operator $L_{P,K}$, 
on the probability that the Hilbert space distance between an arbitrary function $f$ in $\cH$ and 
linear combinations of functions of type $K_{x_i}$, for $(x_i)_i$ sampled independently from $P$, falls below a 
given threshold, see Theorem~\ref{thmbound}.
For sequences of points $(x_i)_{i=1}^\infty$ constituting a 
so-called \emph{uniqueness set}, see Subsection~\ref{ss:usaccp}, the orthogonal projections $\pi^n_x$ to 
$\Span\{K_{x_i} \}^n_{i=1}$ converge in the strong operator topology to the identity operator.
As an application of our main result, we show that, under the assumption that $\cH_P$ is dense in $\cH$, any 
sequence of points sampled independently from $P$ yields a uniqueness set with probability $1$.

The results obtained in this article improve on the results obtained by Kon and Raphael 
in several senses: 
the convergence of approximations is in the RKHS norm, which is stronger than the uniform norm whenever 
the kernel is bounded; 
the type of convergence with respect to the points $(x_i)_i$ is strengthened from convergence in probability to 
almost certain convergence; 
and the separability of $\cH$ then allows the result to be extended from the approximation of a single function 
to the simultaneous approximation of all functions in the Hilbert space. In addition, 
when compared to the existing
methods for this kind of problems, our approach based on the operator $L_{P,K}$ defined at \eqref{e:lpk},
that encodes the interplay
between the kernel $K$ and the probability measure $P$, and the associated RKHS $\cH_P$, is 
completely new and has the potential to overcome many difficulties.

These results are confined to the special
case of a separable RKHS $\cH$ of functions on an arbitrary set $X$, due to several reasons, one of them 
being the fact that the Bochner integral 
is requiring the assumption of separability, but we do not see this as a loss of generality 
since most of the spaces of interest for applications are separable. 
In the last section we present two examples that point out the applicability, and the limitations of our results
as well, the first to the uniform probability distribution on the compact interval $[-\pi,\pi]$, together with a class 
of bounded continuous kernels,
and the second to the Hardy space $H^2(\DD)$ corresponding to the Szeg\"o kernel which is unbounded.
In each case we can explicitly calculate the space $\cH_P$, its reproducing kernel $K_P$, and 
the operator $L_{P,K}$.

\section{Notation and Preliminary Results}\label{s:npr}

\subsection{Reproducing Kernel Hilbert Spaces}\label{ss:rkhs}
In this subsection, we briefly review some concepts and facts on reproducing kernel Hilbert spaces, 
following classical texts such as 
N.~Aronszajn \cite{aron1}, \cite{aron} and L.~Schwartz \cite{schwartz}, or more modern ones such as S.~Saitoh and 
Y.~Sawano \cite[Chapter 2]{saitoh} and V.I.~Paulsen and M.~Raghupathi \cite{paulsen} .

Throughout this article we denote by $\FF$ one of the commutative fields $\RR$ or $\CC$. 
For a nonempty set $X$ let 
$\FF^X$ denote the set of $\FF$-valued functions on $X$, forming an $\FF$-vector space under pointwise addition and scalar multiplication. 
 For each $p \in X$, the \emph{evaluation map} at $p$ is the linear functional
 \begin{equation*}
 \textup{ev}_p \colon \FF^X \to \FF; \ f \mapsto f(p).
 \end{equation*}
The evaluation maps equip $\FF^X$ with the locally convex topology of pointwise convergence, which is the 
weakest topology on $\FF^X$ that renders each evaluation map continuous. Under this topology, a 
generalized sequence in $\FF^X$ converges if and only if it converges pointwise, i.e.\ its image under each 
evaluation map converges. Since each evaluation map is linear and hence the vector space operations are 
continuous, this renders $\FF^X$ into a complete Hausdorff locally convex space. With respect to this 
topology, if $\cX$ is a topological space, a map $\phi \colon \cX \to \FF^X$ is continuous if and only if 
$\textup{ev}_p \circ \phi \colon \cX \to \FF$ is continuous for all $p \in X$.

We are interested in Hilbert spaces $\cH \subseteq \FF^X$ with topologies at least as strong as the
topology of pointwise convergence of
 $\FF^X$, so that the convergence of a sequence of functions in $\cH$ implies that the functions also converge 
pointwise. When $X$ is a finite set, $\FF^X \cong \FF^{d}$, where $d$ is the number of elements of $X$, 
can itself be made into a Hilbert space 
with a canonical inner product $\langle f, g \rangle := \sum_{p \in X} f(p) \overline{g(p)}$, or in general by an
inner product induced by a positive semidefinite $d \times d$ matrix. This leads to the concept of 
reproducing kernel Hilbert spaces.

Recalling the F.~Riesz's Theorem of representations of bounded linear functionals on Hilbert spaces,
if each $\textup{ev}_p : \cH \to \FF$ restricted to $\cH \subseteq \FF^X$ is continuous, for each $p \in X$, then 
there exists a unique vector $K_p \in \cH$ such that $\textup{ev}_p = \langle \cdot, K_p \rangle$. 
But, since each vector in $\cH$ is itself a function $X \to \FF$, 
these vectors altogether define a map $K \colon X \times X \to \FF$, $K(p, q) := K_q(p)$. 
Also, recall that a map $K\colon X\times X\ra \FF$ is usually called a kernel.

\begin{definition}
Let $\cH \subseteq \FF^X$ be a Hilbert space, $K \colon X \times X \to \FF$ a kernel. For each $p \in X$ 
define $K_p := K(\cdot, p) \in \FF^X$. $K$ is said to be a \emph{reproducing kernel} for $\cH$, and $\cH$ is 
then said to be a \emph{reproducing kernel Hilbert space} (RKHS), if, for each $p \in X$, we have
  \begin{itemize}
  \item[(i)] $K_p \in \cH$;
  \item[(ii)] $\textup{ev}_p =  \langle \cdot, K_p \rangle$, that is, for every $f \in \cH$ we have $f(p) =  \langle f, K_p \rangle$.
  \end{itemize}
  The second property is referred to as the \emph{reproducing property} of the kernel $K$.
\end{definition}

We may then summarize the last few paragraphs with the following characterization:
  Let $\cH \subseteq \FF^X$ be a Hilbert space. The following assertions are equivalent:
  \begin{itemize}
  \item[(i)] The canonical injection $i_\cH \colon \cH \to \FF^X$ is continuous.
  \item[(ii)] For each $p \in X$, the map $\textup{ev}_p \colon \cH \to \FF$ is continuous.
  \item[(iii)] $\cH$ admits a reproducing kernel.
  \end{itemize}

In that case, the reproducing kernel admitted by the Hilbert space is unique, by the uniqueness of the 
Riesz representatives $K_p$ of the evaluation maps.
We may further apply the reproducing property to each $K_q$ to
obtain that $K(p,q) = \langle K_q, K_p \rangle$ for each $p, q \in X$, yielding the following properties:
\begin{itemize}
  \item[(i)] For each $p \in X$, $K(p, p) = \|K_p\|^2 \geq 0$.
  \item[(ii)] For each $p, q \in X$, $K(q, p) = \overline{K(p, q)}$ and 
  \begin{equation}\label{e:schwarz}|K(p, q)|^2 \leq K(p,p) K(q,q).\end{equation}
  \item[(iii)] For each $n \in \NN$, $(c_i)_{i=1}^n \in \FF^n$, $(p_i)_{i=1}^n \in X^n$, 
  \begin{equation*}
  \sum_{i=1}^n \sum_{j=1}^n \overline{c_i} c_j K(p_i, p_j) = \|\sum_{i=1}^n c_i K_{p_i}\|^2 \geq 0.
  \end{equation*}
\end{itemize}
The property in \eqref{e:schwarz} is the analogue of the Schwarz Inequality. 
As a consequence of it, if $K(p,p)=0$ for some $p\in X$ then $K(p,q)=K(q,p)=0$ for all $q\in X$.

For any $K \colon X \times X \to \FF$, each $K_p \in \FF^X$ so we may define the subspace 
\begin{equation*}
\widetilde \cH_K := \textup{span} \left\{K_p \mid p \in X\right\}
\end{equation*} of $\FF^X$. If $K$ is the reproducing kernel of a Hilbert space $\cH$, $\widetilde \cH_K$ is 
also a subspace of $\cH$ and \begin{equation*}
\widetilde \cH_K^\perp = \left\{ f \in \cH \mid \forall p \in X, f(p) = \langle f, K_p \rangle = 0 \right\} = \{0\},
\end{equation*}
therefore, $\widetilde \cH_K$ is a dense subspace of $\cH$, equivalently, 
$\left\{K_p \mid p \in X\right\}$ is a total set for $\cH$.

The property at item (iii) is known as the \emph{positive semidefiniteness property}. A positive semidefinite
kernel $K$ is called \emph{definite} if $K(p,p)\neq 0$ for all $p\in X$.  Positive semidefiniteness
is in fact sufficient to characterize all reproducing kernels. By the Moore-Aronszajn Theorem,
for any positive semidefinite kernel $K \colon X \times X \to \FF$, there is a unique Hilbert space 
$\cH_K \subseteq \FF^X$ with reproducing kernel $K$.

Let us briefly recall the construction of the Hilbert space $\cH_K$ in the proof. 
We first render $\widetilde \cH_K$ into a pre-Hilbert space 
satisfying the reproducing property. Define on $\widetilde \cH_K$ the inner product
 \begin{equation*}
  \langle \sum_{i=1}^n a_i K_{p_i}, \sum_{j=1}^m b_j K_{q_j} \rangle_{\widetilde \cH_K} := \sum_{i=1}^n 
\sum_{j=1}^m a_i \overline{b_j} K(q_j, p_i)
 \end{equation*}
 for any $\sum_{i=1}^n a_i K_{p_i}, \sum_{j=1}^m b_j K_{q_j} \in \widetilde \cH_K$. It is proven that the 
definition is correct and provides indeed an inner product. 
 
Let $\widehat \cH_K$ be the completion of $\widetilde \cH_K$, then $\widehat \cH_K$ is a Hilbert space 
with an isometric embedding $\phi \colon \widetilde \cH_K \to \widehat \cH_K$ whose image is dense in 
$\widehat \cH_K$. It is proven that this abstract completion can actually be realized in $\FF^X$ and that it
is the RKHS with reproducing kernel $K$ that we denote by $\cH_K$. 
 
In applications, one of the most useful tool is the interplay between reproducing kernels  and orthonormal 
bases of the underlying RKHSs. Although this fact holds in higher generality, 
we state it for separable Hilbert spaces since, most of the time, this is the case of interest:
letting $\cH\subseteq \FF^X$ be a separable 
RKHS, with reproducing kernel $K$, and let $\{\phi_n\}_n$ be an orthonormal basis 
of $\cH$. Then
\begin{equation}\label{e:kerbasis}
K(p,q)=\sum_{n=1}^\infty \phi_n(p) \ol{\phi_n(q)},\quad p,q\in X,
\end{equation}
where the series converges absolutely pointwise.
 
We now recall a useful result on the construction of new RKHSs and 
positive semidefinite kernels from existing ones. It also shows that the concept of reproducing kernel
Hilbert space is actually a special case of the concept of operator range.
 Let $\cH$ be a Hilbert space, $\phi \colon \cH \to \FF^X$ a continuous linear map. Then $\phi(\cH) \subseteq \FF^X$ with the norm 
 \begin{equation}\label{e:image}
  \|f\|_{\phi(\cH)} := \min \left\{\|u\|_\cH \mid u \in \cH, \ f = \phi(u) \right\} \end{equation}
 is a RKHS, unitarily isomorphic to $(\ker \phi)^\perp$. 
 The kernel for $\phi(\cH)$ is then given by the map 
 \begin{equation}\label{e:kersynth}
 (p, q) \mapsto \langle u_q, u_p \rangle = (\textup{ev}_p \circ \phi)(u_q) = \phi(u_q)(p),\end{equation}
  where $u_q \in \cH$ such that $\textup{ev}_q \circ \phi = \langle \cdot, u_q \rangle$ on $\cH$.
Applying this proposition to particular continuous linear maps, one obtains useful results for pullbacks,
restrictions, sums, scaling, and normalizations of kernels.

\subsection{Integration of RKHS-Valued Functions}\label{ss:intrkhs}
In this article we use integrals of Hilbert space-valued functions. We first provide fundamental 
definitions and properties concerning the \emph{Bochner integral}, an extension of the Lebesgue integral for 
Banach space-valued functions, following D.L~Cohn \cite[Appendix E]{cohn}.

Let $(\cE;\|\cdot\|)$ be a (real or complex) 
Banach space and $(X, \Sigma, \mu)$ a finite measure space. On $\cE$ we consider the Borel 
$\sigma$-algebra denoted by $\cB(\cE)$. A map $f\colon X\ra \cE$ 
is called \emph{measurable} if $f^{-1}(S)\in \Sigma$ for all $S\in \cB(\cE)$ and it is called 
\emph{strongly measurable} if it is measurable and its range $f(X)$ is separable. If $\cE$ is a 
separable Banach space then the concepts coincide. Both sets of measurable functions, 
respectively strongly measurable functions, are vector spaces.
It is proven that, a function $f\colon X\to B$ is strongly measurable if and only if 
there exists a sequence of simple
functions $(\phi_n)_n$ such that $\phi_n\xrightarrow[n]{} f$ pointwise on $X$. In addition, in this case, 
the sequence $(\phi_n)_n$ can be chosen such that  $\|\phi_n(x)\|\leq \|f(x)\|$ for
all $x\in X$. 

A function $f \colon X \to \cE$ is \emph{Bochner integrable} if it is strongly measurable and the scalar
function $X\ni x\mapsto \|f(x)\|\in \RR$ is integrable. In this case, the Bochner integral of $f$ is defined by
approximation with simple functions.
Bochner integrable functions share many properties with scalar-valued integrable functions, but not all.
For example, the collection of all Bochner integrable functions make a vector space and, for any Bochner
integrable function $f$ we have
\begin{equation}\label{e:bochnertrieq}
\left\|\int_X f(x)\de\mu(x)\right\| \leq \int_X \|f(x)\|\de\mu(x).
\end{equation}
Also, letting $L^1(X;\mu;\cE)$ denote the collection of all equivalence classes of Bochner integrable functions,
identified $\mu$-almost everywhere, this is a Banach space with norm
\begin{equation*}
\|f\|_1:=\int_X \|f(x)\|\de\mu(x),\quad f\in L^1(X;\mu;\cE).
\end{equation*}
In addition, the Dominated Convergence Theorem holds for the Bochner integral as well, e.g.\ see 
\cite[Theorem~E.6]{cohn}.
 
In this article, we will use the following result, which is a special case of a theorem of E.~Hille, e.g.\ see
\cite[Theorem~III.2.6]{DiestelUhl}. In Hille's Theorem, the linear transformation is supposed to be only closed 
and, consequently, additional assumptions are needed, so we provide a proof for the special case 
of bounded linear operators for the reader's convenience. 

\begin{theorem}\label{t:hille}
Let $\cE$ be a Banach space, $(X, \mu)$ a measure space, and $f \colon X \to \cE$ a Bochner integrable 
function.
If $L \colon \cE \to \cF$ is a continuous linear transformation between Banach spaces, then 
$L \circ f \colon \cE \to \cF$ is Bochner integrable and 
\begin{equation*} \int_X (L\circ f)(x) \de \mu(x) = L \int_X f(x) \de \mu(x). \end{equation*}
\end{theorem}

\begin{proof} Since $f$ is Bochner integrable, there exists a sequence $(\phi_n)_n$ of simple functions that 
converges pointwise to $f$ on $X$ and $\|\phi_n(x)\|\leq \|f(x)\|$ for all $x\in X$ and all $n\in\NN$. Then,
\begin{equation*}
\|L\phi_n(x)-Lf(x)\|=\|L(\phi_n(x)-f(x))\|\leq \|L\| \|\phi_n(x)-f(x)\|\xrightarrow[n]{}0,\quad x\in X,
\end{equation*}
hence the sequence $(L\circ \phi_n)_n$ converges pointwise to $L\circ f$. Also, it is easy to see that 
$L\circ\phi_n$ is a simple function for all $n\in\NN$. These show that $L\circ f$ is strongly measurable.
Since $\|Lf(x)\|\leq \|L\| \|f(x)\|$ for all $x\in X$ and $f$ is Bochner integrable, it follows that
\begin{equation*}
\int_X \|Lf(x)\|\de\mu(x)\leq \|L\| \int_X \|f(x)\|\de\mu(x)<\infty,
\end{equation*}
hence $L\circ f$ is Bochner integrable.

On the other hand,
\begin{equation*}
\|L\phi_n(x)\|\leq \|L\| \|\phi_n(x)\|\leq \|L\| \|f(x)\|,\quad x\in X,\ n\in\NN,
\end{equation*}
hence, by the Dominated Convergence Theorem for the Bochner integral, it follows that
\begin{align*}
\int_X Lf(x)\de\mu(x) & =\lim_{n\ra\infty} \int_X L\phi_n(x)\de\mu(x) =\lim_{n\ra\infty}L\int_X\phi_n(x)\de\mu(x)\\
& = L\lim_n \int_X \phi_n(x)\de\mu(x)=L\int_X f(x)\de\mu(x).\qedhere
\end{align*}
\end{proof}

A direct consequence of this fact is a sufficient condition for when a pointwise integral coincides with the 
Bochner integral, valid not only for RKHSs but also for Banach spaces of functions on which evaluation
maps at 
any point are continuous, e.g.\ $C(Y)$ for some compact Hausdorff space $Y$.

\begin{proposition}\label{p:bochnerkernel}
Let $(X,\Sigma,\mu)$ be a measure space, $\cB \subseteq \FF^X$ a Banach space of functions on $X$ 
such that all evaluation maps on $\cB$ are continuous. 
Let $\lambda \colon X \times X \to \FF$ be such that for each
$q\in X$ we have $\lambda_q:=\lambda(\cdot,q)\in\cB$. 

If, for each $q \in X$, the map $X\ni q \mapsto \lambda_q\in\cB$ is Bochner integrable, 
then the scalar map $X\ni q \mapsto \lambda(p, q)\in\FF$ is integrable, for each fixed $p \in X$.

Moreover, in that case, the pointwise integral map $X\ni p\mapsto\int_X \lambda(p, q) \de \mu(q)$ lies in $\cB$ 
and coincides with the Bochner integral $\int_X \lambda_q \de \mu(q)$.
\end{proposition}

\begin{proof} Since, for each $q \in X$, the map 
$X\ni q \mapsto \phi(q):=\lambda(\cdot, q)\in\cB$ is 
Bochner integrable, and taking into account that, for all $p\in X$, 
the linear functional $\textup{ev}_p$ is continuous, by Theorem~\ref{t:hille}
we have
\begin{equation*}\label{e:evep} 
\textup{ev}_p \int_X \phi(q) \de \mu(q) = \int_X \textup{ev}_p \circ \phi(q) \de \mu(q). 
 \end{equation*}
Since $\textup{ev}_p \circ \phi(q)=\lambda(p,q)$ for all $p,q\in X$, this means that 
the scalar map $X\ni q \mapsto \lambda(p, q)\in\FF$ is integrable, for each fixed $p \in X$,
and
\begin{equation*}
\textup{ev}_p \int_X \phi(q) \de \mu(q)= \int_X \lambda(p,q)\de\mu(q),\quad p\in X,
\end{equation*}
hence, the pointwise integral map $X\ni p\mapsto\int_X \lambda(p, q) \de \mu(q)$ lies in $\cB$ 
and coincides with the Bochner integral $\int_X \lambda_q \de \mu(q)$.
\end{proof}

\section{Main Results}\label{s:mr}

Throughout this section we consider a probability measure space $(X;\Sigma;P)$ and
a RKHS 
$(\cH;\langle\cdot,\cdot\rangle)$ in $\FF^X$, with norm denoted by $\|\cdot\|_\cH$, 
such that its reproducing kernel $K$ is measurable. In addition, throughout this section, the reproducing
kernel Hilbert space $\cH$ is supposed to be separable.

\subsection{The Reproducing Kernel Hilbert Space $\cH_P$.} 
On the measurable space $(X;\Sigma)$ we define the measure $P_K$
by 
\begin{equation}\label{e:peka}
\de P_K(x)=K(x,x)\de P(x), \quad x\in X\end{equation} 
that is, $P_K$ is the absolutely 
continuous measure with respect to $P$ such that the function
$X\ni x\mapsto K(x,x)$ is the Radon-Nikodym derivative of $P_K$ with respect to $P$.

With respect to the measure space $(X;\Sigma;P_K)$ we consider the Hilbert space $L^2(X;P_K)$. 
Our approach is based on the following natural bounded linear operator mapping $L^2(X;P_K)$ to $\cH$.

\begin{proposition}\label{p:peka} With notation and assumptions as before, let
$\lambda \colon X \to \FF$ be 
a measurable function such that the 
integral $\int_{X} |\lambda(x)|^2 \de P_K(x)$ is finite. Then the Bochner integral 
\begin{equation*}\int_X \lambda(x) K_x \de P(x)\end{equation*} 
exists in $\cH$.

In addition, the mapping
\begin{equation}\label{e:lepeka} 
L^2(X;P_K)\ni \lambda\mapsto L_{P,K}\lambda:=\int_X \lambda(x)K_x\de P(x)\in \cH,
\end{equation} 
is a nonexpansive, hence bounded, linear operator.
\end{proposition}

\begin{proof} By assumptions, the map $X\ni x\mapsto \lambda(x)K_x\in\cH$ is measurable and,
since $\cH$ is separable, it follows that this map is actually strongly measurable.
Letting $\|\cdot\|$ denote the norm on $\cH$ and using the assumption that
$\int_{X} |\lambda(x)|^2 K(x, x) \de P(x)$ is finite, we have
\begin{equation*}  \int_X \|\lambda(x) K_x\|^2_\cH \de P(x) =
\int_{X} |\lambda(x)|^2 K(x, x) \de P(x)  < \infty,
\end{equation*}
hence, by the Schwarz Inequality and taking into account that $P$ is a probability measure, we have
\begin{equation*} 
\int_X \|\lambda(x) K_x\|_\cH \de P(x) \leq \sqrt{\int_X \|\lambda(x) K_x\|^2_\cH \de P(x)} < \infty.
\end{equation*} 
By Theorem~\ref{t:hille} this implies that
the Bochner integral $\int_X \lambda(x) K_x \de P(x)$ exists in $\cH$.
Consequently, the mapping $L_{P,K}$ as in \eqref{e:lepeka} is correctly defined and it is clear that it is a 
linear transformation. 

For arbitrary $\lambda\in L^2(X;P_K)$, by the triangle inequality for the Bochner integral 
\eqref{e:bochnertrieq},
we then have
\begin{align*} \biggl\| \int_X \lambda(x)K_x\de P(x)\biggr\|^2_\cH &
\leq \left( \int_X \|\lambda(x) K_x\| \de P(x) \right)^2 \\ 
& = \biggl( \int_X |\lambda(x)| K(x,x)^{1/2} \de P(x)\biggr)^2 \\
\intertext{and applying the Schwarz Inequality for the integral and taking into account that $P$
is a probability measure}
& \leq \int_X |\lambda(x)|^2 K(x,x)\de P(x)=\|\lambda \|^2_{L^2(X;P_K)},
\end{align*}
hence $L_{P,K}\colon L^2(X;P_K)\ra \cH$ is a nonexpansive linear operator.
\end{proof}

Using the bounded linear operator $L_{P,K}$ defined as in 
\eqref{e:lepeka}, let us denote its range by 
\begin{equation}\label{e:hepe}
\cH_P:=L_{P,K}(L^2(X;P_K)),\end{equation} 
which is a subspace of the RKHS $\cH$.

\begin{proposition}\label{p:hepe} 
$\cH_P$ is a RKHS contained in $\cH$, hence in $\FF^X$, and its reproducing kernel $K_P$ is
\begin{equation*}
K_P(x,y)  = 
\int_X \frac{K(x, z) K(z, y)}{K(z,z)} \de P(z),\quad x,y\in X,
\end{equation*}
where, whenever $K(z,z)=0$,
by convention we define $K(x,z)K(z,y)/K(z,z)=0$ for all $x,y\in X$.
\end{proposition}

\begin{proof}
Since $L^2(X; P_K)$ is a Hilbert space and $L_{P,K}$ is a bounded linear map, 
by \eqref{e:image} it follows that
$\cH_P$ is a RKHS in $\FF^X$,
isometrically isomorphic to the orthogonal complement of $\ker L_{P,K} \subseteq L^2(X; P_K)$,
and its norm is given by
\begin{equation*}
\|g\|_{\cH_P} := \min \left\{ \|\lambda\|_{L^2(X; P_K)} \mid L_{P,K} \lambda = g \right\}, \quad g \in \cH_P.
\end{equation*}

Let 
\begin{equation*}X_0:=\{x\in X\mid K(x,x)=0\},
\end{equation*}
and let us define $u_x\colon X\ra \FF$ by
\begin{equation*}
u_x(y) := \begin{cases} \frac{K(y, x)}{K(y, y)}, & y\in X\setminus X_0,\\
0, & y\in X_0. 
\end{cases}
\end{equation*}
From the Schwarz Inequality for the kernel $K$, it follows that if $x\in X_0$ then $K(x,y)=0$ for all
$y\in X$. This shows that $u_x=0$ for all $x\in X_0$.

For each $x\in X$, by the Schwarz inequality and the fact that $P$ is a 
probability measure we have
\begin{align*}
\int_X |u_x(y)|^2 K(y,y) \de P(y) & = \int_{X\setminus X_0} \frac{|K(y,x)|^2}{K(y,y)} \de P(y) \\
& \leq \int_{X\setminus X_0} \frac{K(y,y)K(x,x)}{K(y,y)} \de P(y)\\
& = K(x,x)P(X\setminus X_0) < \infty,
\end{align*}
hence, $u_x  \in L^2(X, P_K)$. Then, taking into account that $K(x,y)=0$ for all $y\in X_0$ and all $x\in X$,
it follows that, for each $\lambda \in L^2(X, P_K)$ and $x \in X$, we have
\begin{align*}
(L_{P,K} \lambda)(x) & = \int_X \lambda(y) K(x, y) \de P(y)  
 = \int_{X\setminus X_0} \lambda(y) K(x, y) \de P(y) \\
& = \int_{X\setminus X_0} \lambda(y) \overline{\frac{K(y,x)}{K(y,y)}} K(y,y) \de P(y) \\
& = \int_{X} \lambda(y) \overline{u_x(y)} K(y,y) \de P(y)
= \langle \lambda, u_x \rangle_{L^2(X, P_K)}.
\end{align*}
In conclusion, $u_x$ is exactly the representative for the functional $\textup{ev}_x L_{P,K}$ so, 
by \eqref{e:image} the kernel of $\cH_P$ is
\begin{align*}
K_P(x,y)  & = \langle u_y, u_x \rangle_{L^2(X, P_K)} \\
&  = \int_X u_y(z) \overline{u_x(z)} K(z, z) \de P(z) 
= \int_{X\setminus X_0} u_y(z) \overline{u_x(z)} K(z, z) \de P(z) \\
\intertext{and, using the convention that
$K(x,z)K(z,y)/K(z,z)=0$  whenever $K(z,z)=0$ and for arbitrary $x,y\in X$,}
& = \int_X \frac{K(x, z) K(z, y)}{K(z,z)} \de P(z).\qedhere 
\end{align*}
\end{proof}

One of the main results of this article, see Theorem~\ref{thmaeconv}, 
assumes that the space $\cH_P$ 
is dense in $\cH$. The next proposition provides sufficient conditions for this.

\begin{proposition}\label{p:hepedense}
Let $X$ be a topological space, $P$ a Borel probability measure on $X$, $\cH \subseteq \FF^X$ a 
RKHS with measurable kernel $K$, and let $P_K$, $L_{P,K}$ and $\cH_P$ defined as in
\eqref{e:peka}, \eqref{e:lepeka}, and \eqref{e:hepe}, respectively. 

Suppose that $K$ is continuous on $X$, that $\cH \subseteq L^2(X; P_K)$, and that $P$ 
is strictly positive on any nonempty open subset of $X$. Then $\cH_P$ is dense in $\cH$.
\end{proposition}

\begin{proof} The assertion is clearly equivalent with showing that the orthogonal complement of $\cH_P$
in $\cH$ is the null space. To this end,
let $f \in \cH$, $f \perp \cH_P$. That is, for each $\lambda \in L^2(X; P_K)$, we have
\begin{equation*}
\langle f, L_{P,K} \lambda \rangle_\cH = \langle f, \int_X \lambda(x) K_x \de P(x) \rangle = 0.
\end{equation*}
Then noting the fact that $\int \lambda(x) K_x \de P(x)$ is a Bochner integral and hence, 
by Theorem~\ref{t:hille}, it commutes with inner products,
\begin{equation*}
0 = \langle f, \int_X \lambda(x) K_x \de P(x) \rangle = \int_X \ol{\lambda(x)} \langle f, K_x \rangle \de P(x) = \int \ol{\lambda(x)} f(x) \de P(x).
\end{equation*}
By assumption, $f \in \cH \subseteq L^2(X; P_K)$, so we can take $\lambda = f$ to obtain
\begin{equation*}
\int |f(x)|^2 \de P(x) = \int_X \ol{f(x)} f(x) \de P(x) = 0.
\end{equation*}
This implies that $f = 0$ $P$-almost everywhere, i.e.\ the set $f^{-1}(\FF\setminus\{0\})$ has zero $P$ 
measure. 

Since $K$ is continuous by assumption, by the Theorem~2.3 in \cite[Section 2.1.3]{saitoh},
each $f \in \cH$ is continuous hence $f^{-1}(\FF\setminus\{0\})$ is an open subset of $X$. 
But, since $P$ is assumed strictly positive on any nonempty open set, it follows that
$f^{-1}(\FF\setminus\{0\})$ must be empty, hence $f = 0$ identically.
\end{proof}

\subsection{Probability Error Bounds of Approximation}
The first step in our enterprise is to find error bounds for approximations of functions in the reproducing 
kernel Hilbert space $\cH$ in terms of distributional finite linear combinations of functions of type
$K_x$. To do that,
we use the celebrated Markov-Bienaym\'e-Chebyshev Inequality on the 
concentration of probability 
measures to obtain regions of large measure with small approximation error, in terms of the Hilbert space 
norm and not simply the uniform norm.

\begin{theorem}[Markov-Bienaym\'e-Chebyshev's Inequality]\label{c:BCI}
Let $(X;\Sigma;P)$ be a probability space, $(\cB;\|\cdot\|)$ a Banach space, and let
$f,g \colon X \to \cB$ be two Borel measurable functions. Then,
for any $\delta>0$, we have
\begin{equation}\label{e:BCI}
P(\left\{ x \in X \mid \|g(x)\| \geq \delta \right\}) 
\leq \frac{1}{\delta^2} \int_X \|g(x)\|^2 \de P(x).
\end{equation} 
\end{theorem}

The classical Bienaym\'e-Chebyshev Inequality
\begin{equation*}
P(\{x\in X\mid |f(x)-E(f)|\geq k\sigma\})\leq \frac{1}{k^2},
\end{equation*}
is obtained from \eqref{e:BCI} applied for $\cB=\RR$, $g(x)=f(x)-E(f)$, and $\delta=k\sigma$, for 
$k>0$, where $E(f)=\int_X f(x)\de x$ is the expected value 
of the random variable $f$ and
$\sigma^2=E((f-E(f))^2)=E(f^2)-E(f)^2>0$ is the variance of $f$.

\begin{theorem} \label{thmbound}
With notation and assumptions as before, let $\lambda\in L^2(X;P_K)$ and
$f \in \cH$. For each $n \in \NN$ and $\delta>0$, consider the set
\begin{equation}\label{e:aned}
A_{n,\delta}:=\bigl\{(x_1, \ldots, x_n) \in X^n\mid \bigl\|f - \frac{1}{n} 
\sum_{i=1}^n \lambda(x_i) K_{x_i}\bigr\|_\cH\geq \delta\bigr\}.
\end{equation} 
Then, letting $P^n$ denote the product probability measure on $X^n$ and defining the bounded
linear operator $L_{P,K}$ as in \eqref{e:lepeka}, we have
\begin{equation*}
P^n(A_{n,\delta}) \leq \frac{1}{\delta^2} \left\|f - L_{P,K}\lambda \right\|^2_\cH + \frac{1}{n\delta^2} 
\biggl(\|\lambda\|_{L^2(X;P_K)}^2 - \| L_{P,K}\lambda\|^2_\cH\biggr).
\end{equation*}
\end{theorem}

\begin{proof} By Proposition~\ref{p:peka}, the Bochner integral
$\int_X \lambda(x) K_x \de P(x)$ exists  in $\cH$ and 
the linear operator $L_{P,K}$ is well-defined and bounded.
In order to simplify the notation,
considering $g\colon X^n\ra\cH$ the function defined by
\begin{equation*}
g(x_1,\ldots,x_n)=f-\frac{1}{n} 
\sum_{i=1}^n \lambda(x_i) K_{x_i},\quad (x_1,\ldots,x_n)\in X^n,
\end{equation*}
observe that $g$ is measurable and 
for each $\delta > 0$ we have 
\begin{equation}\label{e:anede}A_{n,\delta} = \left\{(x_1, \ldots, x_n) \in X^n \mid \|g(x_1, \ldots, x_n)\|_\cH 
\geq \delta \right\}.\end{equation}
Then we have
\begin{align}
\|g(x_1,\ldots,x_n)\|^2  & = \bigl\|f - \frac{1}{n} \sum_{i=1}^n \lambda(x_i) K_{x_i} \bigr\|^2_\cH \nonumber \\
& = \|f\|^2 - \frac{2}{n} \sum_{i=1}^n \Re \langle f, \lambda(x_i) K_{x_i} \rangle \label{e:eftwo} 
+ \frac{1}{n^2} \sum_{i=1}^n \sum_{j=1}^n \langle \lambda(x_i) K_{x_i}, \lambda(x_j) K_{x_j} \rangle.
\end{align}

Since $P^n$ is a probability measure we have
\begin{equation*}
\int_{X^n} \|f\|^2_\cH \de P^n(x_1, \ldots, x_n) = \|f\|^2_\cH. \end{equation*}
On the other hand,  by Fubini's theorem 
and the fact that the Bochner integral commutes with continuous linear operations, see 
Theorem~\ref{t:hille}, we have
\begin{align*}
\int_{X^n} \Re \langle f, \lambda(x_i) K_{x_i} \rangle \de P^n(x_1, \ldots, x_n) 
  & = \Re \langle f, \int_{X^n} \lambda(x_i) K_{x_i} \de P^n(x_1, \ldots, x_n) \rangle \\
  & = \Re \langle f, \int_{X} \lambda(x) K_{x} \de P(x) \rangle = \Re \langle f, L_{P,K}\lambda \rangle. \end{align*}
Also, for each $i=1,\ldots,n$,
\begin{align*}
\int_{X^n} \langle \lambda(x_i) K_{x_i}, \lambda(x_i) K_{x_i} \rangle \de P^n(x_1, \ldots, x_n) 
  & = \int_{X^n} |\lambda(x_i)|^2 K(x_i, x_i) \de P^n(x_1, \ldots, x_n) \\
  & = \int_{X} |\lambda(x)|^2 K(x, x) \de P(x), \\
  \intertext{and, for each $i, j = 1, \ldots, n$, $i \neq j$,}
\int_{X^n} \langle \lambda(x_i) K_{x_i}, \lambda(x_j) K_{x_j} \rangle \de P^n(x_1, \ldots, x_n) 
  & = \int_{X} \langle \lambda(x_i) K_{x_i}, \int_{X} \lambda(x_j) K_{x_j} \de P(x_j) \rangle \de P(x_i) \\
  & = \langle \int_{X} \lambda(x) K_{x} \de P(x), \int_{X} \lambda(x) K_{x} \de P(x) \rangle \\
  &  = \| \int_{X} \lambda(x) K_{x} \de P(x)\|^2_\cH.
\end{align*}
Integrating both sides of \eqref{e:eftwo} and using all the previous equalities, we therefore have
\begin{align*}
\int_{X^n}  & \|g(x_1, \ldots, x_n)\|^2_\cH   \de P^n(x_1, \ldots, x_n) 
 = \|f\|^2_\cH - \frac{2}{n} \sum_{i=1}^n \Re \langle f, \int_{X} \lambda(x) K_{x} \de P(x) \rangle \\
 & \ \ \ \ \  + \frac{1}{n^2} \sum_{i=1}^n 
\sum_{i\neq j=1}^n \| \int_{X} \lambda(x) K_{x} \de P(x)\|^2_\cH 
 + \frac{1}{n^2} \sum_{i=1}^n \int_{X} |\lambda(x)|^2 K(x, x) \de P(x) \\
& = \|f\|^2_\cH - 2 \Re \langle f, \int_{X} \lambda(x) K_{x} \de P(x) \rangle 
+ \frac{n-1}{n} \| \int_{X} \lambda(x) K_{x} \de P(x)\|^2_\cH \\ 
& \phantom{\|f\|^2 - 2 \Re \langle f, \int_{X} \lambda(x) K_{x} \de P(x) \rangle} 
+ \frac{1}{n} \int_{X} |\lambda(x)|^2 K(x, x) \de P(x) \\
& = \left\|f \! -\! \int_{X}\! \lambda(x) K_{x} \de P(x)
 \right\|^2_\cH 
 + \frac{1}{n} \biggl( \int_{X} \! |\lambda(x)|^2 K(x, x) \de P(x) - \| \int_{X}\! \lambda(x) K_{x} \de P(x)
\|^2_\cH \biggr)\\
& = \left\|f  - L_{P,K}\lambda
 \right\|^2_\cH 
 + \frac{1}{n} \biggl( \int_{X}  |\lambda(x)|^2 K(x, x) \de P(x) - \| L_{P,K}\lambda\|^2_\cH \biggr).
\end{align*}
Finally, in view of the Markov-Bienaym\'e-Chebyshev Inequality as
in \eqref{e:BCI}, when $X$ is replaced by $X^n$ and $P$ by $P^n$, and  
taking into account the previous equality and \eqref{e:anede}, we get
\begin{align*} 
P^n(A_{n,\delta}) & \leq \frac{1}{\delta^2} \int_{X^n} \|g(x_1, \ldots, x_n)\|^2_\cH \de P^n(x_1, \ldots, x_n) \\
& = \frac{1}{\delta^2} \left\|f - L_{P,K}\lambda \right\|^2_\cH 
+ \frac{1}{n\delta^2} \biggl(\|\lambda\|_{L^2(X;P_K)}^{2} - \| L_{P,K}\lambda \|^2_\cH\biggr),
\end{align*} 
which is the required inequality.
\end{proof}

\subsection{Convergence in Probability}\label{ss:cp}
As with the special case of kernel embeddings, for which $\lambda = 1$, see Smola et al. \cite{smola},
we may use the bound in Theorem~\ref{thmbound} to obtain a statement of convergence in probability.

With notation and assumptions as before, given $f \in \cH$ and fixed $(x_i)_{i=1}^N \in X$, the 
problem of finding the optimal $(\omega^N_i(f))_{i=1}^N \in \FF^N$ to minimize $\|f - \sum_{i=1}^N \omega^N_i(f) 
K_{x_i}\|_\cH$ is straightforward: $\sum_{i=1}^N \omega^N_i(f) K_{x_i}$ is the orthogonal projection of $f$ to 
$\textup{span}\{K_{x_i}\}_{i=1}^N$. 

We may assume without loss of generality that $\{K_{x_i}\}_{i=1}^N$ are linearly independent, by removing 
points as necessary without affecting $\textup{span}\{K_{x_i}\}_{i=1}^N$ (or losing any information about $f$, 
since $\sum_{i=1}^N c_i K_{x_i} = 0$ implies $\sum_{i=1}^N \overline{c_i} f(x_i) = 0$ by the reproducing 
property).  According to H.~K\"orezlio\u glu \cite{korezli}, if $(x_i)_{i=1}^N \in X$ is a  sampling
such that $\{K_{x_i}\}_{i=1}^N$ are linearly independent and considering the
finite-dimensional subspace $\cH^N_x := \textup{span}\{K_{x_i}\}_{i=1}^N$ of $\cH$, then the orthogonal 
projection $\pi^N_x$ of $\cH$ onto $\cH^N_x$ is given by
\begin{equation*}
\pi^N_x(f) = \sum_{i=1}^N \omega^\pi_i(f) K_{x_i} := \sum_{i=1}^N \sum_{j=1}^N f(x_j) \Gamma^N_{ji} K_{x_i} 
= \sum_{i=1}^N \sum_{j=1}^N \langle f, K_{x_j} \rangle \Gamma^N_{ji} K_{x_i}
\end{equation*}
for any $f \in \cH$, where $\Gamma^N \in \cM_N(\FF)$ is the inverse of the \emph{Gram matrix} 
$G^N := \left[ K(x_j, x_i) \right]_{i,j=1}^N = \left[ \langle K_{x_i}, K_{x_j} \rangle  \right]_{i,j=1}^N$ of 
$\{x_1, \ldots, x_N\}$.

More generally, if $\{K_{x_i}\}_{i=1}^N$ are not linearly independent, for any subset $s = (x_{i_j})_{j=1}^K$ 
such that $\{K_{x_{i_j}}\}_{j=1}^K$ form a basis for $\cH^N_x$, we have $\cH^N_x = \cH^K_s$ and
\begin{equation*}\pi^N_x = \pi^K_s = \sum_{j=1}^K \sum_{k=1}^K \langle \cdot, K_{x_{i_k}} \rangle 
\Gamma^K_{kj} K_{x_{i_j}}.\end{equation*}
Note that, in general, $\omega^\pi_i$ is not simply a multiple of $f(x_i)$ 
hence, setting $\omega_i := V_i f(x_i)$ for any fixed $V_i$ will not yield the best possible 
approximation. However, with such coefficients dependent only on $x_i$, it will be easier
to bound $\|f - \sum_i \omega_i K_{x_i}\|$ across different $(x_i)_i$s than $\|f - \pi^N_x f\|$. Then any upper 
bound on $\|f - \sum_i \omega_i K_{x_i}\|$ for some fixed $(\omega_i)_i$ will also be an upper bound on 
$\|f - \pi^N_x f\|$.

\begin{theorem}[Convergence in Probability of Projections] \label{thmprobconv}
Let $X$, $P$, $K$, and $\cH$  be as in Theorem~\ref{thmbound}.  For each sequence 
$x=(x_i)_i\in X^\NN$ and each $n\in\NN$, 
let $\pi^n_x$ denote the orthogonal projection of $\cH$ onto 
$\Span\{K_{x_i}\}_{i=1}^n$.
Let $f \in \cH$ and, for each $\delta > 0$ and $n \in \NN$,
define
\begin{equation*}
B_{n, \delta} := \left\{(x_1, \ldots, x_n) \in X^n \mid  \left\| f - \pi^n_x f \right\|_\cH \geq \delta \right\}.
\end{equation*}
Then, for each $\delta>0$
\begin{equation*}
\limsup_{n \to \infty} P^n(B_{n, \delta}) \leq \frac{1}{\delta^2} d_\cH(f, \cH_P)^2,
\end{equation*}
where $d_\cH(f, \cH_P) = \inf_{g \in \cH_P} \|f-g\|$. 

In particular, if $f$ belongs to $\overline{\cH_P}^\cH$, the closure of  $\cH_P$ with respect to the topology of 
$\cH$, then
\begin{equation*}
\lim_{n \to \infty} P^n(B_{n, \delta}) = 0.
\end{equation*}
\end{theorem}

\begin{proof}
Let $\lambda \in L^2(X, P_K)$ and fix $\delta>0$, arbitrary. 
Then 
\begin{equation}\label{e:fepin}
\|f - \pi^n_x f \|_\cH \leq \left\|f - \frac{1}{n} \sum_{i=1}^n \lambda(x_i) K_{x_i} \right\|_\cH,
\end{equation} hence, with 
notation as in \eqref{e:aned}, we have
$B_{n,\delta}\subseteq A_{n,\delta}$.  By Theorem \ref{thmbound}, this implies
\begin{equation*}
P^n(B_{n, \delta}) \leq \frac{1}{\delta^2} \left\|f - L_{P,K} \lambda \right\|^2_\cH + \frac{1}{n\delta^2} \left[ \|\lambda\|_{L^2(X, P_K)}^2 - \|L_{P,K} \lambda\|^2_\cH\right].
\end{equation*}
Therefore,
\begin{align*}
\limsup_{n \to \infty} P^n(B_{n, \delta}) 
& \leq \limsup_{n \to \infty} \left[ \frac{1}{\delta^2} \left\|f - L_{P,K} \lambda \right\|_\cH^2 + \frac{1}{n\delta^2} \left( \|\lambda\|_{L^2(X, P_K)}^2 - \|L_{P,K} \lambda\|_\cH^2 \right) \right] \\
& = \frac{1}{\delta^2} \left\|f - L_{P,K} \lambda \right\|_\cH^2.
\end{align*}
Thus, since the left-hand side is independent of $\lambda$,
\begin{equation*}
\limsup_{n \to \infty} P^n(B_{n, \delta}) \leq 
\inf_{\lambda \in L^2(X, P_K)} \frac{1}{\delta^2} \left\|f - L_{P,K} \lambda 
\right\|_\cH^2 = \frac{1}{\delta^2} d_\cH(f, \cH_P)^2.
\end{equation*}
In particular, if
$f$ belongs to $\overline{\cH_P}^\cH$, then $d_\cH(f, \cH_P) = 0$. 
\end{proof}

\subsection{Uniqueness Sets and Almost Certain Convergence of Projections.}\label{ss:usaccp}

With notation and assumptions as before, we now follow \cite[Subsection~2.4.4]{saitoh} in 
recalling the strong convergence of $\pi^N_x$ to the identity map as 
$N \to \infty$ for appropriately chosen $(x_i)_{i=1}^\infty$.
Since $\cH$ is separable, there exists a countable subset of $\{K_p\}_{p \in X}$ which is total in $\cH$; thus, 
there exists a countable set $F \subseteq X$ such that $\textup{span}\{K_x\}_{x \in F}$ is dense in $\cH$.
This motivates the following definition:
a countable subset $\{x_i\}_{i=1}^\infty$ of $X$ is called a \emph{uniqueness set} for $\cH$ if 
$\{K_{x_i}\}_{i=1}^\infty$ is a total set in $\cH$, that is, if $f \in \cH$ such that
$f(x_i) = 0$ for all $i \in \NN$ implies $f = 0$.
Then, the so-called Ultimate Realization of RKHSs, cf.\ \cite[Theorem 2.33]{saitoh}, reads as follows:
if $\{x_i\}_{i=1}^\infty$ is a uniqueness set such that $\{K_{x_i}\}_{i=1}^\infty$ is linearly independent, 
$G^N$ is the Gram matrix for $\{x_i\}_{i=1}^N$, $\Gamma^N = (G^N)^{-1}$, then for each $f \in \cH$,
\begin{equation*}
\lim_{N \to \infty} \pi^N_x f = \lim_{N \to \infty} \sum_{i=1}^N \sum_{j=1}^N f(x_i) \Gamma^N_{ij} K_{x_j} = f
\end{equation*}
under the topology of $\cH$, with distance decreasing monotonically. Consequently,
\begin{equation*}
\langle f, g \rangle = \lim_{N \to \infty} \sum_{i=1}^N \sum_{j=1}^N f(x_i) \Gamma^N_{ij} \overline{g(x_j)}
\end{equation*}
for $f, g \in \cH$, and
\begin{equation*}
f(x) = \langle f, K_x \rangle = \lim_{N \to \infty} \sum_{i=1}^N \sum_{j=1}^N f(x_i) \Gamma^N_{ij} K(x, x_j)
\end{equation*}
for $f \in \cH, x \in X$.
This has implications in interpolation theory, e.g.\ see \cite[Corollary~2.6]{saitoh}.

Coming back to our problem, 
by noting that $\|f - \pi^n_x f \|$, unlike $\left\|f - \frac{1}{n} \sum_{i=1}^n \lambda(x_i) K_{x_i} \right\|$, is 
monotonically nonincreasing with respect to $n$, our next goal is to
strengthen Theorem~\ref{thmprobconv} to almost certain convergence after passing to a single 
measure space. Firstly, recall that, 
e.g.\ see \cite[Proposition 10.6.1]{cohn}, 
the countably infinite product space $X^\NN$ equipped with the smallest 
$\sigma$-algebra rendering each projection map $X_i \colon X^\NN \to X$ measurable admits a unique 
probability measure $P^\NN$ such that the projection maps are independent random variables with distribution 
$P$.

\begin{lemma}\label{l:penn}
Let $X$, $P$, $K$, and $\cH$ be as in Theorem \ref{thmbound} and $f \in \cH$. For each $\delta > 0$ define
\begin{equation*}
S_{n, \delta} := \left\{ x = (x_k)_{k=1}^\infty \in X^\NN \mid \|f - \pi^n_x f \|_\cH \geq \delta \right\}, \quad n\in\NN,
\end{equation*}
and 
\begin{equation}\label{e:sedelta}
S_{\delta} := \left\{ x = (x_k)_{k=1}^\infty \in X^\NN \mid \forall N \in \NN, \exists n \geq N, \|f - \pi^n_x f \|_\cH 
\geq \delta \right\} = \bigcap_{N \in \NN} \bigcup_{n \geq N} S_{n, \delta}.
\end{equation}
Then,
\begin{equation*}
P^\NN(S_{\delta}) \leq \frac{1}{\delta^2} d_\cH(f, \cH_P)^2,
\end{equation*}
and, consequently, if $f \in \overline{\cH_P}^\cH$, then
\begin{equation*}
P^\NN(S_{\delta}) = 0.
\end{equation*}
\end{lemma}

\begin{proof}
Observe that for each $n, m \in \NN$ such that $n > m$, $\|f - \pi^n_x f\|_\cH \leq \|f - \pi^m_x f\|_\cH$, for each 
$x \in X^\NN$, and hence $S_{n, \delta} \subseteq S_{m, \delta}$ for each $\delta > 0$. Then,
\begin{equation*}
S_{\delta} = \bigcap_{N \in \NN} \bigcup_{n \geq N} S_{n, \delta} = \bigcap_{N \in \NN} S_{N, \delta},
\end{equation*}
hence, for any $\lambda \in L^2(X, P_K)$,
\begin{equation*}
P^\NN(S_\delta) \leq \inf_{N \in \NN} P^\NN(S_{N, \delta}) \leq \frac{1}{\delta^2} \|f - L_{P,K} \lambda \|_\cH^2,
\end{equation*}
since $P^\NN$ is monotone and $S_\delta \subseteq S_{N, \delta}$ for all $N \in \NN$.
\end{proof}

The main result of this subsection is the following

\begin{theorem}[Almost Certain Convergence of Projections] \label{thmaeconv}
Let $X, P, K, \cH$ be as in Theorem~\ref{thmbound} and 
suppose $\cH_P$ is dense in $\cH$. Then, for each $f \in \cH$,
\begin{equation*}
P^\NN\left(\left\{ x \in X^\NN \mid \pi^n_x f \xrightarrow[n]{} f \right\}\right) = 1,
\end{equation*}
hence,
\begin{equation*}
P^\NN\left(\left\{ x \in X^\NN \mid \forall f \in \cH, \pi^n_x f \xrightarrow[n]{} f \right\}\right) = 1.
\end{equation*}
\end{theorem}

\begin{proof}
Let $f \in \cH$. With the same sets $S_\delta$ defined in \eqref{e:sedelta},
\begin{align*}
\left\{ x \in X^\NN \mid \pi^n_x f \not\to f \right\} &
= \left\{ x \in X^\NN \mid \exists \delta > 0, \forall N \in \NN, \exists n \geq N, \|f - \pi^n_x f \|_\cH 
\geq \delta \right\}
\\ & = \bigcup_{\delta > 0} S_\delta.
\end{align*}
Observe further that $S_\delta \subseteq S_{\delta^\prime}$ whenever $\delta > \delta^\prime$, 
and for each $\delta > 0$ there exists $m \in \NN$ such that $\delta > 1/m$, so that
\begin{equation*}
\left\{ x \in X^\NN \mid \pi^n_x f \not\xrightarrow[n]{} f \right\} 
= \bigcup_{0 < \delta \leq 1} S_\delta = \bigcup_{m \in \NN} S_{1/m}
\end{equation*}
thus, taking into account that $\cH_P$ is dense in $\cH$ and using Lemma~\ref{l:penn}, we get
\begin{equation*}
P^\NN\left(\left\{ x \in X^\NN \mid \pi^n_x f \not\xrightarrow[n]{} f \right\}\right) \leq \sum_{m \in \NN} P^
\NN(S_{1/m}) = \sum_{m \in \NN} 0 = 0.
\end{equation*}

Since $\cH$ is separable let $\cD$ be a countable dense subset of $\cH$. Since each $\pi^n_x$ is 
a continuous linear operator with operator norm $1$, $\pi^n_x f \to f$ for all $f \in \cH$ iff $\pi^n_x f \to f$ for all 
$f \in \cD$. Thus by the countable subadditivity of $P^\NN$,
\begin{align*}
P^\NN\left(\left\{ x \in X^\NN \mid \exists f \in \cH, \pi^n_x f \not\xrightarrow[n]{} f \right\}\right) 
& = P^\NN\left(\left\{ x \in X^\NN \mid \exists f \in \cD, \pi^n_x f \not\xrightarrow[n]{} f \right\}\right) \\
& = P^\NN\left(\bigcup_{f \in \cD} \left\{ x \in X^\NN \mid \pi^n_x f \not\xrightarrow[n]{} f \right\}\right) \\
& \leq \sum_{f \in \cD} P^\NN\left(\left\{ x \in X^\NN \mid \pi^n_x f \not\xrightarrow[n]{} f \right\}\right)=0. \qedhere
\end{align*}
\end{proof}

In summary, for a given probability measure $P$ under the assumption that it 
renders the space $\cH_P$, the image of $L_{P,K}$, dense in $\cH$, a 
sequence of points sampled independently from $P$ yields a uniqueness set with probability $1$. 
Proposition shows a sufficient condition, valid for many applications, 
when this assumption holds.

\section{Examples}\label{s:ex}

In this final section we provide detailed examples of applicability of the results on approximation error bounds
obtained in the previous section.

\subsection{Uniform distribution on a compact interval}\label{ss:udoci}
Let $(\mu_j)_{j\in\ZZ} \in l_1(\ZZ)$ be such that $\mu_j > 0$ for all $j \in \ZZ$ and denote 
$\mu:=\sum_{j\in \ZZ}\mu_j$. 
For each $j \in \ZZ$ define 
\begin{equation*}\phi_j\colon[-\pi,\pi]\ra \CC,\quad
\phi_j (t) := \emath^{\iac  \pi j t}, \quad t\in[-\pi,\pi],
\end{equation*} and consider the Hilbert space
\begin{equation*}
\cH = \left\{ \sum_{j\in\ZZ} c_j \phi_j \mid \sum_{j\in\ZZ} \frac{|c_j|^2}{\mu_j} < \infty \right\},
\end{equation*}
with the inner product
\begin{equation*}
\langle \sum_{j\in\ZZ} c_j \phi_j, \sum_{j\in\ZZ} d_j \phi_j \rangle = \sum_{j\in\ZZ} \frac{c_j \overline{d_j}}{\mu_j}.
\end{equation*}
Then $\{\sqrt{\mu_j}\phi_j\}_{j\in\ZZ}$ is an orthonormal basis of $\cH$ and, for an arbitrary function $f\in\cH$,
we have the Fourier representation
\begin{equation}\label{e:fot}
f(t)=\sum_{j\in\ZZ} c_j \phi_j(t),\quad t\in [-\pi,\pi],
\end{equation}
with coefficients $\{c_j\}_{j\in\ZZ}$ subject to the condition 
\begin{equation}\label{e:fehad}
\|f\|_\cH^2:=\sum_{j\in\ZZ} \frac{|c_j|^2}{\mu_j} < \infty,\end{equation} 
where the convergence of the series from 
\eqref{e:fot} is at least guaranteed with respect to the norm $\|\cdot\|_\cH$. 
However, for any $m\in\NN_0$ and $t\in[-\pi,\pi]$, by the Cauchy inequality we have
\begin{equation*}
\sum_{|j|\geq m} |c_j \phi_j(t)|\leq \bigl(\sum_{|j|\geq m} \frac{|c_j|^2}{\mu_j}\bigr)^{1/2} \bigl( \sum_{|j|\geq m}
\mu_j\bigr)^{1/2}\xrightarrow[m\ra\infty]{} 0,
\end{equation*}
hence the convergence in \eqref{e:fot} is absolutely and uniformly on $[-\pi,\pi]$, in particular $f$ is 
continuous.

By \eqref{e:kerbasis} $\cH$ has the reproducing kernel 
\begin{equation}\label{e:kasete}
K(s,t)=\sum_{j\in\ZZ} \mu_j \emath^{\iac\pi j(s-t)} = \sum_{j\in\ZZ} \mu_j \phi_j(s) \overline{\phi_j(t)},
\end{equation}
and the convergence of the series is guaranteed at least pointwise. In addition, for any $t\in [-\pi,\pi]$
we have
\begin{equation*}
K(t,t)=\sum_{j\in\ZZ} \mu_j |\phi_j(t)|^2=\sum_{j\in\ZZ}\mu_j=\mu,
\end{equation*}
and hence the kernel $K$ is bounded. In particular, this implies that, actually, the series in 
\eqref{e:kasete} converges absolutely and uniformly on $[-\pi,\pi]$, hence the kernel $K$ is continuous
on $[-\pi,\pi]\times [-\pi,\pi]$.
That is, $K(s,t)$ is given by $\kappa(s-t)$ where $\kappa \colon \RR \to \CC$ is a continuous function with 
period $2\pi$ whose Fourier coefficients $(\mu_j)_{j\in\ZZ}$ are all positive and absolutely summable.

Let $P$ be the normalized Lebesgue measure on $[-\pi,\pi]$, equivalently, the
uniform probability distribution 
on $[-\pi,\pi]$, and observe that $\{\phi_j\}_{j\in\ZZ}$
is an orthonormal basis of the Hilbert space $L_P[-\pi,\pi]$. With notation as in \eqref{e:peka}, we
have $\de P_K(t)=K(t,t)\de P(t)=\mu\de P(t)$ hence $L^2_{P_K}[-\pi,\pi]=L^2_P[-\pi,\pi]$ with
norms differing by multiplication with $\mu>0$. In particular, $\{\phi_j/\sqrt\mu\}_{j\in\ZZ}$ is an orthonormal basis
of the Hilbert space $L^2_{P_K}[-\pi,\pi]$.

We consider now the nonexpansive 
operator $L_{P,K}\colon L^2_{P_K}[-\pi,\pi]\ra \cH$ defined as in \eqref{e:lepeka}.
Then, for any $j\in\ZZ$ and $t \in [-\pi,\pi]$, we have
\begin{align*}
(L_{P,K} \phi_j)(t) 
& = \int_{-\pi}^\pi \phi_j(s) K(t, s) \de P(s) = \int_{-\pi}^\pi \phi_j(s) \left( \sum_{k\in\ZZ} \mu_k \phi_k(t) 
\overline{\phi_k(s)} \right) \de P(s) \\
& = \sum_{k\in\ZZ} \mu_k \phi_k(t) \int_{-\pi}^\pi \phi_j(s) \overline{\phi_k(s)} \de P(s) = \sum_{k\in\ZZ} \mu_k 
\phi_k(t) \delta_{jk} = \mu_j \phi_j(t),
\end{align*}
where, the series commutes with the integral either by the Bounded Convergence Theorem for the Lebesgue
integral, or by using the uniform convergence of the series and the Riemann integral.
Similarly, the Hilbert space $\cH_P:=L_{P,K}(L^2_{P_K}[-\pi,\pi])$, as in Proposition~\ref{p:hepe},
is a RKHS, with kernel,
\begin{align*}
K_P(s,t) & = \frac{1}{2\pi} \int_{-\pi}^\pi \frac{\left(\sum_{j\in\ZZ} \mu_j \phi_j(s) \overline{\phi_j(z)}\right)\left(\sum_{l\in\ZZ} \mu_l \phi_l(z) \overline{\phi_l(t)}\right)}{\sum_{j\in\ZZ} \mu_j} \de z \\
& = \frac{1}{\mu} \sum_{j\in\ZZ} \sum_{l\in\ZZ} \mu_j \mu_l \phi_j(s) \overline{\phi_l(t)} \frac{1}{2\pi} \int_{-\pi}^\pi \phi_j(z) \overline{\phi_l(z)} \de z \\
& = \frac{1}{\mu} \sum_{j\in\ZZ} \sum_{l\in\ZZ} \mu_j \mu_l \phi_j(s) \overline{\phi_l(t)} \delta_{jl} 
= \sum_{j\in\ZZ} \frac{\mu_j^2}{\mu} \phi_j(s) \overline{\phi_j(t)}.
\end{align*}
Thus, letting $\mu^\prime_j := \frac{\mu_j^2}{\mu} \leq \mu_j$, $j\in\ZZ$ and noting that 
$\sum_{j\in\ZZ} \mu^\prime_j \leq \sum_{j\in\ZZ} \mu_j < \infty$, we have
\begin{equation*}
\cH_P = \left\{ \sum_{j\in\ZZ} c_j \phi_j \mid \sum_{j\in\ZZ} \frac{|c_j|^2}{\mu^\prime_j} < \infty \right\} 
= \left\{ \sum_{j\in\ZZ} c_j \phi_j \mid \sum_{j\in\ZZ} \frac{|c_j|^2}{\mu_j^2} < \infty \right\}.
\end{equation*}
In particular, $\cH_P$ is dense in $\cH$ since both contain $\textup{span}\{\phi_j\}_{j\in\ZZ}$ as dense
subsets, but this follows from the more general statement in Proposition~\ref{p:hepedense} as well.

Let now $\lambda\in L^2_{P_K}[-\pi,\pi]=L^2_P[-\pi,\pi]$ be arbitrary, hence
\begin{equation*}
\lambda =\sum_{j\in\ZZ}\lambda_j \phi_j,\quad \sum_{j\in\ZZ}|\lambda_j|^2<\infty,\quad 
\|\lambda\|^2_{L^2_{P_K}[-\pi,\pi]}=\frac{1}{\mu}\sum_{j\in\ZZ}|\lambda_j|^2.
\end{equation*}
Then,
\begin{equation*}
(L_{P,K}\lambda)(t)=\bigl(L_{P,K}\sum_{j\in\ZZ}\lambda_j\phi_j\bigr)(t)
=\sum_{j\in\ZZ}\lambda_j\mu_j \phi_j(t),\quad t\in[-\pi,\pi],
\end{equation*}
and, consequently,
\begin{equation*}
\|L_{P,K}\lambda\|^2_\cH=\sum_{j\in\ZZ}\frac{|\lambda_j|^2\mu_j^2}{\mu_j}=\sum_{j\in\ZZ} \mu_j |\lambda_j|^2.
\end{equation*}
Also, for arbitrary $f\in \cH$ as in \eqref{e:fot} and \eqref{e:fehad}, we have
\begin{equation*}
\|f-L_{P,K}\lambda\|_\cH^2 = \bigl\| \sum_{j\in\ZZ}(c_j-\lambda_j\mu_j)\phi_j\bigr\|^2_\cH=\sum_{j\in\ZZ}
\frac{|c_j-\lambda_j \mu_j|^2}{\mu_j}.
\end{equation*}

Let $(x_n)_{n\in\NN}$ be a sequence of points in $[-\pi,\pi]$. By Theorem~\ref{thmbound} and
taking into account of the inequality \eqref{e:fepin},
for any $N\in\NN$ and $\delta>0$ we have
\begin{align} \label{e:penebi}
P^N\bigl( \|f-\pi^N_x f\|_\cH\geq \delta\bigr) & \leq
P^N\bigl( \|f-\frac{1}{N} \sum_{n=1}^N \lambda(x_n)K_{x_n}\|_\cH\geq \delta\bigr) \\
& \leq \frac{1}{\delta^2}\sum_{j\in\ZZ} \frac{|c_j-\lambda_j\mu_j|^2}{\mu_j} +
\frac{1}{N\delta^2} \biggl(\sum_{j\in\ZZ}(\mu-\mu_j)|\lambda_j|^2\biggr). \nonumber
\end{align} 
On the other hand, we observe that in the inequality \eqref{e:penebi} the left hand side does not depend
on $\lambda$ and hence, for any $\epsilon>0$ there exists $\lambda\in L^2_{P_K}[-\pi,\pi]$ such that 
\begin{equation*}
P^N\bigl( \|f-\pi^N_x f\|_\cH\geq \delta\bigr) < \frac{\epsilon}{2} +
\frac{1}{N\delta^2} \biggl(\sum_{j\in\ZZ}(\mu-\mu_j)|\lambda_j|^2\biggr),
\end{equation*}
and then, for sufficiently large $N$ we get
\begin{equation*}
P^N\bigl( \|f-\pi^N_x f\|_\cH\geq \delta\bigr) < \epsilon.
\end{equation*}

In particular, if $f\in\cH_P$, that is, the inequality \eqref{e:fehad} is replaced by the stronger one
\begin{equation*}
\sum_{j\in\ZZ} \frac{|c_j|^2}{\mu_j^2} < \infty,
\end{equation*}
we can choose $\lambda_j=c_j/\mu_j$, $j\in\ZZ$, and 
we have $\lambda\in L^2_{P_K}[-\pi,\pi]$,
hence
\begin{align*}
P^N\bigl( \|f-\pi^N_x f\|_\cH\geq \delta\bigr) & \leq 
\frac{1}{N\delta^2} \biggl(\sum_{j\in\ZZ}\frac{(\mu-\mu_j)|c_j|^2}{\mu_j^2}\biggr).
\end{align*}
For example, this is the case for $f=\phi_k$ for some $k\in\ZZ$, hence $c_j=\delta_{j,k}$, $j\in\ZZ$, and letting
$\lambda=\phi_k/\mu_k$, hence $\lambda_j=\delta_{j,k}/\mu_j$, $j\in\ZZ$, we have $f=L_{P,K}\lambda$ and
hence,
\begin{equation*}
P^N(\|\phi_k-\pi_x^N\phi_k\|\geq \delta)\leq \frac{1}{N\delta^2\mu_k^2}\sum_{\ZZ\ni j\neq k} \mu_j.
\end{equation*}
This shows that, the larger $\mu_k$ is, the faster $\phi_k$ will be approximated but, since 
$\mu_j \xrightarrow[j]{} 0$, $\phi_j$s cannot be approximated uniformly, in the sense that there does 
not exist a single $N$ to make each $\left\|\phi_j - \pi^N_x \phi_j \right\|_\cH$ bounded by the same 
$\delta$ with the same probability $\eta$. 

This analysis can be applied more generally to kernels that admit an expansion analogous to \eqref{e:kasete} 
under basis functions $(\phi_j)_j$ which constitute a total orthonormal set in $L^2(X; P_K)$, e.g.\ as 
guaranteed by Mercer's Theorem \cite[Theorem 2.30]{saitoh}.

\subsection{The Hardy space $H^2(\DD)$}\label{ss:hardy} 
We consider the open unit disc in the complex plane 
$\DD=\{z\in\CC\mid |z|<1\}$ and the Szeg\"o kernel
\begin{equation}\label{e:hardykernel}
K(z,\zeta)=\frac{1}{1-z\ol\zeta}=\sum_{n=0}^\infty z^n\ol{\zeta}^n,\quad z,\zeta\in\DD,
\end{equation}
where the series converges absolutely and uniformly on any compact subset of $\DD$.
The RKHS associated to $K$ is the Hardy space $H^2(\DD)$ of all functions
$f\colon \DD\ra \CC$ that are holomorphic in $\DD$ with power series expansion
\begin{equation}\label{e:fez}
f(z)=\sum_{n=0}^\infty f_n z^n,
\end{equation}
such that the coefficients sequence $(f_n)_n$ is in $\ell^2_\CC(\NN_0)$. 
The inner product in $H^2(\DD)$ is
\begin{equation*}
\langle \sum_{n=0}^\infty f_nz^n,\sum_{n=0}^\infty g_nz^n\rangle=\sum_{n=0}^\infty f_n\ol{g_n},
\end{equation*}
with norm
\begin{equation*}
\|\sum_{n=0}^\infty f_nz^n\|^2=\sum_{n=0}^\infty |f_n|^2.
\end{equation*}
For each $\zeta\in\DD$ we have
\begin{equation*}
\|K_\zeta\|=\bigl(\sum_{n=0}^\infty |\zeta|^{2n}\bigr)^{1/2}=\frac{1}{\sqrt{1-|\zeta|^2}},
\end{equation*}
hence the kernel $K$ is unbounded.

We consider $P$ the normalized Lebesgue measure on $\DD$, that is, for $z=x+\iac y=r\emath^{\iac\theta}$
we have
\begin{equation*}
\de P(z)=\frac{1}{\pi} \de A(x,y)=\frac{r}{\pi}\de\theta\de r,
\end{equation*}
hence,
\begin{equation*}
\de P_K(z)=\frac{r}{\pi(1-r^2)}\de\theta\de r.
\end{equation*}
Then, $L^2(\DD;P_K)$ is contractively  embedded in $L^2(\DD;P)$.

Further on, in view of Proposition~\ref{p:hepe} and \eqref{e:hardykernel}, for any $z,\zeta\in\DD$ we have
\begin{align}\label{e:kapez}
K_P(z,\zeta) & =\frac{1}{\pi}\int_0^1\int_0^{2\pi} 
\frac{r(1-r^2)}{(1-zr\emath^{-\iac\theta})(1-\ol{\zeta}r\emath^{\iac\theta})}\de\theta\de r \\
& =  \frac{1}{\pi}\int_0^1\int_0^{2\pi} \sum_{n=0}^\infty\sum_{k=0}^\infty (1-r^2)r^{n+k+1}
\emath^{\iac(n-k)\theta}
z^n\ol{\zeta}^k \de\theta\de r \nonumber \\ 
\intertext{which, by using twice the Bounded Convergence Theorem for the Lebesgue measure, equals}
& = \sum_{n=0}^\infty\sum_{k=0}^\infty \frac{1}{\pi}\int_0^1\int_0^{2\pi} (1-r^2)r^{n+k+1}\emath^{\iac(n-k)\theta} 
\de\theta\de r z^n\ol{\zeta}^k \nonumber \\
& = \sum_{n=0}^\infty 4\int_0^1(1-r^2)r^{2n+1}\de r  z^n\ol{\zeta}^n \nonumber \\
& = \sum_{n=0}^\infty \frac{z^n\ol{\zeta}^n}{(n+1)(n+2)}. \nonumber
\end{align}
This shows that the RKHS $H^2_P(\DD)$ induced by $K_P$ 
consists of all functions $h$ that are holomorphic
in $\DD$ with power series representation $h(z)=\sum_{n=0}^\infty h_n z^n$ and such that
\begin{equation*}
\sum_{n=0}^\infty (n+1)(n+2)|h_n|^2<\infty.
\end{equation*}
In particular, an orthonormal basis of $H^2_P(\DD)$ is $\{z^n/\sqrt{(n+1)(n+2)}\}_{n\geq 0}$ and hence 
$H^2_P(\DD)$ is dense in the Hardy space $H^2(\DD)$.

In order to calculate the operator $L_{P,K}\colon L^2(\DD;P_K)\ra H^2(\DD)$, let $\lambda\in L^2(\DD;P_K)$ 
be arbitrary, that is, $\lambda$ is a complex valued measurable function on $\DD$ such that
\begin{equation}\label{e:nolam}
\|\lambda\|_{L^2(\DD;P_K)}^2=
\frac{1}{\pi}\int_0^1\int_0^{2\pi} \frac{|\lambda(r\emath^{\iac\theta})|^2 r}{1-r^2}\de\theta\de r<\infty.
\end{equation}
Then, in view of Proposition~\ref{p:bochnerkernel}, we have
\begin{align}\label{e:lepekal}
(L_{P,K}\lambda)(z) & = \frac{1}{\pi}\int_0^1\int_0^{2\pi} \lambda(r\emath^{\iac\theta}) K(z,r\emath^{\iac\theta})
r\de\theta\de r \\
& = \frac{1}{\pi}\int_0^1\int_0^{2\pi} \lambda(r\emath^{\iac\theta}) \sum_{n=0}^\infty z^n r^{n+1}
\emath^{-\iac n\theta}\de\theta\de r \nonumber \\
\intertext{which, by the Bounded Convergence Theorem, equals}
& = \sum_{n=0}^\infty \frac{1}{\pi}\int_0^1\int_0^{2\pi} \lambda(r\emath^{\iac\theta})r^{n+1}
\emath^{-\iac n \theta} \de\theta\de r z^n = \sum_{n=0}^\infty \lambda_n z^n,\nonumber
\end{align}
where, for each integer $n\geq 0$ we denote
\begin{equation}\label{e:lan}
\lambda_n= \frac{1}{\pi}\int_0^1\int_0^{2\pi} \lambda(r\emath^{\iac\theta})r^{n+1}
\emath^{-\iac n \theta} \de\theta\de r.
\end{equation}
Observing that, letting $\phi_n(z):=\sqrt{n+1}z^n$, for all integer $n\geq 0$ and $z\in\DD$, the set
$\{\phi_n\}_{n\geq 0}$ is orthonormal in $L^2(\DD;P)$, it follows that 
$\lambda_n=\langle \lambda,\phi_n\rangle_{L^2(\DD;P)}$ for all integer $n\geq 0$ and, hence, 
$(\lambda_n)_{n\geq 0}$ is the weighted
sequence of Fourier coefficients of $\lambda$ with respect to the system of orthonormal functions 
$\{\phi_n\}_{n\geq 0}$ in $L^2(\DD;P)$. On the other hand, since $L^2(\DD;P_K)$ is contractively 
embedded in $L^2(\DD;P)$, this shows that
$L_{P,K}$ is the restriction to $L^2(\DD;P_K)$ of
a Bergman type weighted projection of $L^2(\DD;P)$ onto a subspace of the Hardy space $H^2(\DD)$, 
that happens to be exactly $H^2_P(\DD)$.

Finally, let $f\in H^2(\DD)$ with power series representation as in \eqref{e:fez}
and let $\lambda\in L^2(\DD;P_K)$ with norm given as in
\eqref{e:nolam}. Then, by Theorem~\ref{thmbound} and
taking into account of the inequality \eqref{e:fepin}, for any $N\in\NN$ and $\delta>0$ we have
\begin{align}\label{e:penepi}
P^N\bigl( \|f-\pi_\mathbf{z}^N f\|_{H^2(\DD)}\geq \delta\bigr) & \leq
P^N\bigl(\|f-\frac{1}{N}\sum_{i=1}^N \lambda(z_i)K_{z_i}\|_{H^2(\DD)}\geq \delta\bigr) \\
& \leq \frac{1}{\delta^2} \sum_{n=0}^\infty |f_n-\lambda_n|^2 +\frac{1}{N\delta^2}\bigl( 
\|\lambda\|_{L^2(\DD;P_K)}^2-\sum_{n=0}^\infty |\lambda_n|^2\bigr),\nonumber
\end{align}
where $\mathbf{z}=(z_i)_{i\in\NN}$ denotes 
an arbitrary sequence of points in $\DD$ and $\pi^N_\mathbf{z}$ denotes the projection of $H^2(\DD)$ onto
$\Span\{K_{z_i}\mid i=1,\ldots,N\}$.
By exploiting the fact that the left hand side in \eqref{e:penepi} does not depend on $\lambda$ and
the density of $H^2_P(\DD)$ in $H^2(\DD)$, for any 
$\varepsilon>0$ there exists $\lambda\in L^2(\DD;P_K)$ such that
\begin{equation*}
P^N \left( \|f-\pi_\mathbf{z}^N f\|_{H^2(\DD)}\geq \delta \right) \leq \frac{\varepsilon}{2} +\frac{1}{N\delta^2}\bigl( 
\|\lambda\|_{L^2(\DD;P_K)}^2-\sum_{n=0}^\infty |\lambda_n|^2\bigr),
\end{equation*}
and hence, for $N$ sufficiently large, we have
\begin{equation*}
P^N\left( \|f-\pi_\mathbf{z}^N f\|_{H^2(\DD)}\geq \delta\right) \leq \varepsilon.
\end{equation*}

Let us consider now the special case when the function $f\in H^2_P(\DD)$, that is, 
with respect to the representation as in \eqref{e:fez},  we have the stronger condition
\begin{equation*}
\sum_{n=0}^\infty (n+1)(n+2)|f_n|^2<\infty.
\end{equation*}
In this case, letting
\begin{equation*}
\lambda(z):=\sum_{n=0}^\infty (n+1)(n+2)(1-|z|^2)f_n z^n,\quad z\in\DD,
\end{equation*}
calculations similar to \eqref{e:kapez} and \eqref{e:lepekal} show that
\begin{equation*}
\frac{1}{\pi} \int_0^1 \int_0^{2\pi} \frac{|\lambda(r\emath^{\iac\theta})|^2r}{1-r^2}\de\theta\de r
= \sum_{n=0}^\infty (n+1)(n+2)|f_n|^2<\infty,
\end{equation*}
hence $\lambda\in L^2(\DD;P_K)$, and
\begin{equation*}
(L_{P,K}\lambda)(z)=\frac{1}{\pi}\int_0^\infty \int_0^{2\pi} \lambda(r\emath^{\iac\theta})K(z,r\emath^{\iac\theta})r
\de\theta\de r = \sum_{n=0}^\infty f_n z^n=f(z),\quad z\in \DD,
\end{equation*}
hence, the first term in the right hand side of \eqref{e:penepi} vanishes and we get
\begin{equation*}
P^N\left( \|f-\pi_\mathbf{z}^N f\|_{H^2(\DD)}\geq\delta\right) 
\leq \frac{1}{N\delta^2} \sum_{n=0}^\infty (n^2+3n+1)|f_n|^2.
\end{equation*}
For example, if $f(z)=z^n$ for some integer $n\geq 0$, then
\begin{equation*}
P^N\left( \|f-\pi_\mathbf{z}^N f\|_{H^2(\DD)}\geq \delta\right) \leq \frac{n^2+3n+1}{N\delta^2},
\end{equation*}
showing that better approximations are obtained for smaller $n$ than for bigger $n$.

\section{Conclusions}

Certain key properties of Hilbert spaces drive the analysis that has been obtained in this article, as well as the 
properties of reproducing kernel Hilbert spaces that render them attractive for function approximation. 
The Hilbert space structure provides orthogonal projections as the unique best approximation, which can 
be computed using the reproducing property as an exact interpolation, and are shown to converge 
monotonically to the function for uniqueness sets. The monotonicity of 
convergence is then used to derive almost certain convergence directly from convergence in probability, 
and thus establish sufficient conditions for almost every sequence of samples 
from a probability distribution to be a uniqueness set. For the approximation bound itself, stated in 
Theorem \ref{thmbound}, the mean squared distance in Chebyshev's inequality can be calculated explicitly 
thanks to the norm being induced by an inner product and the existence of the Bochner integral. 

We did not include in this article an example with the Gaussian kernel, one of the most useful kernels in 
applications, although calculations similar to those obtained in Section~\ref{s:ex} are available. 
One of the reasons for this omission
is that the Gaussian kernels have additional invariance and differentiability/analyticity properties that can be 
used in order to provide stronger results by using slightly different techniques that are in progress and will 
make the contents of a future research.
\bigskip

The authors of this article declare no conflict of interests.

\end{document}